\documentclass{article}
\usepackage{amssymb,amsmath,amsthm}
\usepackage{latexsym}

\numberwithin{equation}{section}

\newtheorem{theorem}           {Theorem}      [section]
\newtheorem{definition+} {Definition}      [section]
\newtheorem{lemma}  {Lemma}  [section]
\newtheorem{corollary}  {Corollary} [section]
\newtheorem{proposition+}  {Proposition} [section]
\newtheorem{example+}  {Example}  [section]
\newtheorem{remark+}  {Remark}  [section]
\newtheorem{problem+}  {Problem}  [section]

\newenvironment{definition}{\begin{definition+}\rm}{\end{definition+}\rm}

\newenvironment{example}{\begin{example+}\rm}{\end{example+}\rm}
\newenvironment{remark}{\begin{remark+}\rm}{\end{remark+}\rm}

\makeatletter\renewcommand{\section}{\@startsection{section}{1}{0pt}%
{-3.5ex plus -1ex minus-.2ex}{2.3ex plus .2ex}%
{\normalfont\bfseries}}\makeatother

\makeatletter\renewcommand{\subsection}{\@startsection{subsection}{1}{0pt}%
{-3.25ex plus -1ex minus-.2ex}{1.5ex plus .2ex}%
{\normalfont\bfseries}}\makeatother

\begin{document}

\begin{center}
{\large\bf MINIMIZING MEASURES ON CONDENSERS\\ OF~INFINITELY MANY
PLATES}\makeatletter{\renewcommand*{\@makefnmark}{}
\footnotetext{{\it {\rm 2000} Mathematics Subject Classification}:
31C15.}\makeatother}\makeatletter{\renewcommand*{\@makefnmark}{}
\makeatother}
\end{center}

\begin{center}
{\it Natalia Zorii}
\end{center}

\begin{abstract}
The study deals with the interior capacities of condensers in a
locally compact space, a condenser being treated here as a
countable, locally finite collection of closed sets~$A_i$, $i\in I$,
with the sign~$+1$ or~$-1$ prescribed such that the oppositely
signed sets are mutually disjoint. We are concerned with the minimal
energy problem over the class of linear combinations $\sum_{i\in
I}\,({\rm sign}\,A_i)\,\mu^i$, where $\mu^i$ is a nonnegative Radon
measure supported by~$A_i$ and normalized by $\int g_i\,d\mu^i=a_i$,
$g_i$ and $a_i$ being given. For all positive definite kernels
satisfying Fuglede's condition of consistency between the strong and
vague ($={}$weak$*$) topologies, we establish sufficient conditions
for the existence of minimizers and analyze properties of their
uniqueness, compactness, and continuity.\end{abstract}

\section{Introduction}The present work is devoted to further development of
the theory of interior capacities of condensers in a locally compact
space. A condenser will be treated here as a countable, locally
finite collection~$\mathbf A$ of closed sets $A_i$, $i\in I$, with
the sign~$+1$ or~$-1$ prescribed such that the oppositely signed
sets are mutually disjoint.

For a background of the theory for condensers
of finitely many plates we refer the reader to~\cite{GR0,GR},
\cite[Chap.~VIII]{ST}, and \cite{Z2}--\cite{Z6}; see
also~\cite[Chap.~5]{NS} and~\cite{O}, where the condensers were
additionally assumed to be compact.

The reader is expected to be familiar with the principal notions and
results of the theory of measures and integration on a locally
compact space; its exposition can be found in~\cite{B2,B3,E2} (see
also~\cite{F1,Z3} for a brief survey).

In all that follows, $\mathrm X$ denotes a locally compact Hausdorff
space and $\mathfrak M=\mathfrak M(\mathrm X)$ the linear space of
all real-valued Radon measures~$\nu$ on~$\mathrm X$ equipped with
the {\it vague\/} ($={}${\it weak}$*$) topology, i.\,e., the
topology of pointwise convergence on the class $\mathrm C_0(\mathrm
X)$ of all real-valued continuous functions~$\varphi$ on~$\mathrm X$
with compact support.

A {\it kernel\/}~$\kappa$ on $\mathrm X$ is meant to be an element
from $\mathrm\Phi(\mathrm X\times\mathrm X)$, where
$\mathrm\Phi(\mathrm Y)$ consist of all lower semicontinuous
functions~$\psi:\mathrm Y\to(-\infty,\infty]$ such that
$\psi\geqslant0$ unless $\mathrm Y$ is compact. Given
$\nu,\,\nu_1\in\mathfrak M$, the {\it mutual energy\/} and the {\it
potential\/} with respect to a kernel~$\kappa$ are defined
respectively by
$$\kappa(\nu,\nu_1):=\int\kappa(x,y)\,d(\nu\otimes\nu_1)(x,y)\quad\mbox{and}\quad
\kappa(\,\cdot\,,\nu):=\int\kappa(\,\cdot\,,y)\,d\nu(y).$$ (Here and
in the sequel, when introducing notation, we always tacitly assume
the corresponding object on the right to be well defined.) For
$\nu=\nu_1$ the mutual energy $\kappa(\nu,\nu_1)$ defines the {\it
energy\/} of~$\nu$. Let $\mathcal E$ consist of all $\nu\in\mathfrak
M$ with $-\infty<\kappa(\nu,\nu)<\infty$.

We shall be mainly concerned with a {\it positive definite\/}
kernel~$\kappa$, which means that it is symmetric (i.\,e.,
$\kappa(x,y)=\kappa(y,x)$ for all $x,\,y\in\mathrm X$) and the
energy $\kappa(\nu,\nu)$, $\nu\in\mathfrak M$, is nonnegative
whenever defined. Then $\mathcal E$ forms a pre-Hil\-bert space with
the scalar product $\kappa(\nu,\nu_1)$ and the seminorm
$\|\nu\|_\mathcal E:=\sqrt{\kappa(\nu,\nu)}$ (see~\cite{F1}). A
positive definite kernel~$\kappa$ is called {\it strictly positive
definite\/} if the seminorm $\|\,\cdot\,\|_\mathcal E$ is a norm.

Let $\mathfrak M^+(E)$ consist of all nonnegative measures
$\nu\in\mathfrak M$ supported by~$E$, where $E$ is a given closed
subset of~$\mathrm X$, and let $\mathcal E^+(E):=\mathfrak
M^+(E)\cap\mathcal E$. We also write $\mathfrak M^+:=\mathfrak
M^+(\mathrm X)$ and $\mathcal E^+:=\mathcal E^+(\mathrm X)$.

Given a condenser $\mathbf{A}=(A_i)_{i\in I}$, we consider the class
$\mathfrak M(\mathbf{A})$ of all {\it linear combinations\/}
$\mu=\sum_{i\in I}\,\alpha_i\mu^i$, where $\alpha_i:={\rm
sign}\,A_i$ and $\mu^i\in\mathfrak M^+(A_i)$, equipped with a
relation of identity and a topology so that it becomes {\it
homeomorphic\/} to the product space $\prod_{i\in I}\,\mathfrak
M^+(A_i)$ (where every $\mathfrak M^+(A_i)$ is endowed with the
vague topology). We call the corresponding topology on $\mathfrak
M(\mathbf{A})$ the $\mathbf{A}$-{\it vague\/} topology.

To introduce a proper notion of {\it energy\/} $\kappa(\mu,\mu)$ of
$\mu\in\mathfrak M(\mathbf{A})$, we observe that, due to the local
finiteness of a condenser, there is a unique Radon measure
$R\mu\in\mathfrak M$ such that $R\mu(\varphi)=\sum_{i\in
I}\,\alpha_i\mu^i(\varphi)$ for all $\varphi\in\mathrm C_0(\mathrm
X)$. Therefore, it is reasonable to set
$\kappa(\mu,\mu):=\kappa(R\mu,R\mu)$. This notion can {\it
equivalently\/} be defined also by
\begin{equation*}\label{vectoren}\kappa(\mu,\mu)=\sum_{i,j\in
I}\,\alpha_i\alpha_j\,\kappa(\mu^i,\mu^j)\end{equation*}
(see~Sec.~\ref{sec:ensum}), which is in agreement with an
electrostatic interpretation of a condenser. Let $\mathcal
E(\mathbf{A})$ consist of all $\mu\in\mathfrak M(\mathbf{A})$ whose
energy is finite.

Having fixed a vector-valued function $\mathbf{g}=(g_i)_{i\in I}$,
where all $g_i:\mathrm X\to(0,\infty)$ are continuous, and a
numerical vector $\mathbf{a}=(a_i)_{i\in I}$ with $a_i>0$, we next
define the {\it interior capacity\/} of a condenser~$\mathbf{A}$
(with respect to~$\kappa$, $\mathbf{a}$, and~$\mathbf{g}$) as
$1\bigl/\inf\,\kappa(\mu,\mu)$, the infimum being taken over all
$\mu\in\mathcal E(\mathbf{A})$ normalized by $\int g_i\,d\mu^i=a_i$,
\mbox{$i\in I$}. Along with its electrostatic interpretation, such a
notion has found various important applications to approximation
theory, geometrical function theory, and potential theory itself
(see~the books~\cite{H,NS,ST} and the references cited therein).

The main question we shall be interested in is whether
minimizers~$\lambda_{\mathbf{A}}$ in the above minimal energy
problem exist. If $\mathbf{A}$ is finite, all $A_i$ are compact,
while $\kappa(x,y)$ is continuous on $A_i\times A_j$ whenever
$\alpha_i\ne\alpha_j$, then the existence of
those~$\lambda_{\mathbf{A}}$ can be easily established by exploiting
the $\mathbf{A}$-vague topology only (see~\cite[Th.~2.30]{O};
cf.~also~\cite{GR0,GR,NS,ST}, related to the logarithmic kernel in
the plane). However, the question becomes rather nontrivial if any
of these three assumptions is dropped.

To solve the problem on the existence of
minimizers~$\lambda_{\mathbf{A}}$ in the general case where
$\mathbf{A}$ is infinite and (or) noncompact, we restrict ourselves
to positive definite kernels and work out an approach based on the
following arguments.

The set $\mathcal E(\mathbf{A})$ forms a semimetric space with the
semimetric \[\|\mu_1-\mu_2\|_{\mathcal
E(\mathbf{A})}:=\|R\mu_1-R\mu_2\|_\mathcal E=\Bigl[\sum_{i,j\in
I}\,\alpha_i\alpha_j\,\kappa(\mu^i_1-\mu^i_2,\mu^j_1-\mu^j_2)\Bigr]^{1/2},\]
and, for rather general $\kappa$, $\mathbf{g}$, and~$\mathbf{a}$,
the topological subspace of $\mathcal E(\mathbf{A})$ consisting of
all $\mu$ with $\int g_i\,d\mu^i\leqslant a_i$, $i\in I$, is {\it
complete\/} (see Sec.~\ref{sec:strong}).

Using these arguments, we obtain sufficient conditions for the
existence of minimizers~$\lambda_{\mathbf{A}}$ and establish
statements on their uniqueness, $\mathbf{A}$-vague compactness, and
continuity under exhaustion of~$\mathbf{A}$ by~$\mathbf{K}$ with
compact~$K_i$, $i\in I$. See~Sec.~\ref{sec:5}.

The results obtained and the approach applied develop and generalize
the corresponding ones from the author's articles~\cite{Z3,Z4,Z6},
related to the condensers of finitely many plates.

\section{Preliminaries: topologies, consistent and perfect
kernels}\label{sec:2}

From now on, the kernel under consideration is always assumed to be
positive definite. In addition to the {\it strong\/} topology
on~$\mathcal E$, determined by the seminorm
$\|\cdot\|:=\|\cdot\|_\mathcal E$, it is often useful to consider
the {\it weak\/} topology on~$\mathcal E$, defined by means of the
seminorms $\nu\mapsto|\kappa(\nu,\mu)|$, $\mu\in\mathcal E$
(see~\cite{F1}). The Cauchy-Schwarz inequality
\begin{equation*}
|\kappa(\nu,\mu)|\leqslant\|\nu\|\,\|\mu\|,\quad\mbox{where
$\nu,\,\mu\in\mathcal E$,}\end{equation*} implies immediately that
the strong topology on $\mathcal E$ is finer than the weak one.

In~\cite{F1,F2}, B.~Fuglede introduced the following two {\it
equivalent\/} properties of consistency between the induced strong,
weak, and vague topologies on~$\mathcal E^+$:
\begin{itemize}
\item[\rm(C$_1$)] {\it Every strong Cauchy net in
$\mathcal E^+$ converges strongly to every its vague cluster point;}
\item[\rm(C$_2$)] {\it Every strongly bounded and vaguely convergent net in
$\mathcal E^+$ converges weakly to the vague limit.}
\end{itemize}

\begin{definition}
Following Fuglede~\cite{F1}, we call a kernel~$\kappa$ {\it
consistent} if it satisfies either of the properties~(C$_1$)
and~(C$_2$), and {\it perfect\/} if, in addition, it is strictly
positive definite.\end{definition}

\begin{remark} One has to consider {\it nets\/} or {\it filters\/}
in~$\mathfrak M^+$ instead of sequences, since the vague topology in
general does not satisfy the first axiom of countability. We follow
Moore's and Smith's theory of convergence, based on the concept of
nets (see~\cite{MS}; cf.~also~\cite[Chap.~0]{E2} and
\cite[Chap.~2]{K}). However, if $\mathrm X$ is metrizable and
countable at infinity, then $\mathfrak M^+$ satisfies the first
axiom of countability (see~\cite[Lemma~1.2.1]{F1}) and the use of
nets may be avoided.\end{remark}

\begin{theorem}[Fuglede \cite{F1}]\label{th:1} A kernel $\kappa$ is perfect if and
only if $\mathcal E^+$ is strongly complete and the strong topology
on~$\mathcal E^+$ is finer than the vague one.\end{theorem}

\begin{example} In $\mathbb R^n$, $n\geqslant 3$, the
Newtonian kernel $|x-y|^{2-n}$ is perfect~\cite{Car}. So are the
Riesz kernel $|x-y|^{\alpha-n}$, $0<\alpha<n$, in~$\mathbb R^n$,
$n\geqslant2$~\cite{D1, D2}, and the restriction of the logarithmic
kernel $-\log\,|x-y|$ in $\mathbb R^2$ to an open unit ball
(see~\cite{L}). Furthermore, if $D$ is an open set in~$\mathbb R^n$,
$n\geqslant 2$, and its generalized Green function~$g_D$ exists
(see, e.\,g.,~\cite[Th.~5.24]{HK}), then $g_D$ is perfect as
well~\cite{E1}.\end{example}

\begin{remark} As is seen from the above definitions and Theorem~\ref{th:1}, the concept of consistent
or perfect kernels is an efficient tool in minimal energy problems
over {\it nonnegative\/} Radon measures with finite energy. Indeed,
the theory of capacities of {\it sets\/} has been developed
in~\cite{F1} exactly for those kernels. We shall show that this
concept is efficient, as well, in minimal energy problems over
(finite or infinite) {\it linear combinations\/} $\mu\in\mathcal
E(\mathbf{A})$, $\mathbf{A}$ being a condenser. This is guaranteed
by the completeness of proper topological subspaces of~$\mathcal
E(\mathbf{A})$, established in Sec.~\ref{sec:strong}
below.\end{remark}

\section{Condensers, associated measures, energies and
potentials}\label{sec:3} We start by introducing and discussing
basic concepts of the theory of interior capacities of condensers,
some of which have already been briefly mentioned in the
Introduction (cf.~also~\cite{Z.arh}).

\subsection{Condensers}
Let $I^+$ and $I^-$ be countable (finite or infinite) disjoint sets
of indices~$i\in\mathbb N$, where the latter is allowed to be empty,
and let $I$ denote their union. Assume that to every $i\in I$ there
corresponds a nonempty, closed set~$A_i\subset\mathrm X$.

\begin{definition} A collection $\mathbf{A}=(A_i)_{i\in I}$ is
called an $(I^+,I^-)$-{\it condenser\/} (or simply a {\it
condenser\/}) in~$\mathrm X$ if every compact subset of~$\mathrm X$
intersects with at most finitely many~$A_i$ and
\begin{equation}
A_i\cap A_j=\varnothing\quad\mbox{for all \ } i\in I^+, \ j\in I^-.
\label{non}
\end{equation}\end{definition}

The sets $A_i$, $i\in I^+$, and $A_j$, $j\in I^-$, are called the
{\it positive\/} and, respectively, {\it negative plates\/} of the
con\-den\-ser $\mathbf A$. Note that any two equally sign\-ed plates
can intersect each other. Given $I^+$ and $I^-$, let $\mathfrak
C=\mathfrak C(I^+,I^-)$ be the class of all
$(I^+,I^-)$-con\-den\-sers in~$\mathrm X$. A condenser
$\mathbf{A}\in\mathfrak C$ is said to be {\it compact\/} if so are
all~$A_i$, $i\in I$, and {\it finite\/} if $I$ is finite. Also the
following notation will be used:
\begin{equation*}
A^+:=\bigcup_{i\in I^+}\,A_i,\qquad A^-:=\bigcup_{i\in I^-}\,A_i.
\end{equation*} Observe that $A^+$
and~$A^-$ might both be noncompact even for a compact~$\mathbf{A}$.

\subsection{Measures associated with a condenser. $\mathbf{A}$-vague topology}
With the preceding notation, write
\begin{equation*}
\alpha_i:=\left\{
\begin{array}{rll} +1 & \mbox{if} & i\in I^+,\\ -1 & \mbox{if} & i\in
I^-.\\ \end{array} \right. \end{equation*} Given
$\mathbf{A}\in\mathfrak C$, let $\mathfrak M(\mathbf{A})$ consist of
all (finite or infinite) {\it linear combinations\/}
\begin{equation*}\mu:=\sum_{i\in I}\,\alpha_i\mu^i,\quad\mbox{where
\ }\mu^i\in\mathfrak M^+(A_i).\end{equation*} Any two $\mu_1$ and
$\mu_2$ in $\mathfrak M(\mathbf{A})$ are regarded to be {\it
identical\/} ($\mu_1\equiv\mu_2$) if and only if $\mu_1^i=\mu_2^i$
for all $i\in I$. Then, under the relation of identity thus defined,
the following correspondence is {\it one-to-one\/}:
\begin{equation*}\mathfrak M(\mathbf{A})\ni\mu\mapsto(\mu^i)_{i\in
I}\in \prod_{i\in I}\mathfrak M^+(A_i).\end{equation*} We call
$\mu\in\mathfrak M(\mathbf{A})$ a {\it measure associated
with\/}~$\mathbf A$, and $\mu^i$ its $i$-{\it coordinate}.

For measures associated with a condenser, it is therefore natural to
introduce the following concept of convergence, actually
corresponding to the vague convergence by coordinates. Let $S$
denote a directed set of indices, and let $\mu_s$, $s\in S$, and
$\mu_0$ be given elements of the class~$\mathfrak M(\mathbf A)$.

\begin{definition} A net $(\mu_s)_{s\in S}$ is said to converge to
$\mu_0$ $\mathbf A$-{\it va\-gue\-ly\/} if
\begin{equation*}\mu^i_s\to\mu_0^i\quad\mbox{vaguely for all \ }i\in
I.\end{equation*}\end{definition}

Then $\mathfrak M(\mathbf A)$, equipped with the topology of
$\mathbf A$-vague convergence, becomes {\it homeomorphic\/} to the
product space $\prod_{i\in I}\,\mathfrak M^+(A_i)$, where every
$\mathfrak M^+(A_i)$ is endowed with the vague topology. Since
$\mathfrak M(\mathrm X)$ is Hausdorff, so are both the spaces
$\mathfrak M(\mathbf A)$ and $\prod_{i\in I}\mathfrak M^+(A_i)$
(see, e.\,g.,~\cite[Chap.~3, Th.~5]{K}).

Similarly, a set $\mathfrak F\subset\mathfrak M(\mathbf A)$ is
called $\mathbf A$-{\it vaguely bounded\/} if all its
$i$-projections are vaguely bounded~--- that is, if for every
$\varphi\in\mathrm C_0(\mathrm X)$ and every $i\in I$,
\[\sup_{\mu\in\mathfrak F}\,|\mu^i(\varphi)|<\infty.\]

\begin{lemma}\label{lem:vaguecomp} Any $\mathbf A$-vaguely bounded part of $\mathfrak
M(\mathbf A)$ is $\mathbf A$-vaguely relatively compact.\end{lemma}

\begin{proof} Since by~\cite[Chap.~III, \S~2, Prop.~9]{B2} any
vaguely bounded and closed part of~$\mathfrak M$ is vaguely compact,
the lemma follows from Tychonoff's theorem on the product of compact
spaces (see, e.\,g.,~\cite[Chap.~5, Th.~13]{K}).\end{proof}

\subsection{Mapping $R:\mathfrak M(\mathbf{A})\to\mathfrak M$.
Relation of equivalency on $\mathfrak M(\mathbf{A})$}
Since each compact subset of~$\mathrm X$ intersects with at most
finitely many~$A_i$, for every $\varphi\in\mathrm C_0(\mathrm X)$
only a finite number of~$\mu^i(\varphi)$ (where $\mu\in\mathfrak
M(\mathbf{A})$ is given), are nonzero. This yields that to every
$\mu\in\mathfrak M(\mathbf{A})$ there corresponds a unique Radon
measure~$R\mu$ such that
\begin{equation*}
R\mu(\varphi)=\sum_{i\in I}\,\alpha_i\mu^i(\varphi)\quad\mbox{for
all \ }\varphi\in\mathrm C_0(\mathrm X);
\end{equation*}
due to (\ref{non}), positive and negative parts in Jordan's
decomposition of~$R\mu$ can respectively be written in the form
\begin{equation*} R\mu^+=\sum_{i\in I^+}\,\mu^i\quad\mbox{and}\quad
R\mu^-=\sum_{i\in I^-}\,\mu^i.\end{equation*}

Of course, the mapping~$R:\mathfrak M(\mathbf{A})\to\mathfrak M$
thus defined is in general non-injective, i.\,e., one may choose
$\mu_1,\,\mu_2\in\mathfrak M(\mathbf{A})$ so that
$\mu_1\not\equiv\mu_2$, while $R\mu_1=R\mu_2$. We call
$\mu_1,\,\mu_2\in\mathfrak M(\mathbf{A})$ {\it equivalent
in\/}~$\mathfrak M(\mathbf{A})$, and write $\mu_1\cong\mu_2$, if
their $R$-images coincide~--- or, which is equivalent, whenever
$\sum_{i\in I}\,\mu_1^i=\sum_{i\in I}\,\mu_2^i$.

Observe that the relation of equivalency in $\mathfrak
M(\mathbf{A})$ implies that of identity (and, hence, these two
relations on~$\mathfrak M(\mathbf{A})$ are actually equivalent) if
and only if all~$A_i$, $i\in I$, are mutually disjoint.

\begin{lemma}\label{lem:vague} The $\mathbf A$-vague convergence of
$(\mu_s)_{s\in S}$ to~$\mu_0$ implies the vague convergence of
$(R\mu_s)_{s\in S}$ to~$R\mu_0$.\end{lemma}

\begin{proof} This is obvious since the support of any
$\varphi\in\mathrm C_0(\mathrm X)$ might have points in common with
only a finite number of~$A_i$.\end{proof}

\begin{remark} Lemma~\ref{lem:vague} in general can not be
inverted. However, if all $A_i$, $i\in I$, are mutually disjoint,
then the vague convergence of $(R\mu_s)_{s\in S}$ to~$R\mu_0$
implies the $\mathbf A$-vague convergence of $(\mu_s)_{s\in S}$
to~$\mu_0$. This can be seen by using the Tietze-Urysohn extension
theorem (see, e.\,g.,~\cite[Th.~0.2.13]{E2}).\end{remark}

\subsection{Energies and potentials of measures associated with a
condenser}\label{sec:ensum} To introduce energies and potentials of
linear combinations $\mu\in\mathfrak M(\mathbf{A})$, we start with
the following two lemmas, the former one being well known (see,
e.\,g., \cite{F1}).

\begin{lemma}\label{lemma:lower}
Let $\mathrm Y$ be a locally compact Hausdorff space. If
$\psi\in\mathrm\Phi(\mathrm Y)$ is given, then the map
$\nu\mapsto\int\psi\,d\nu$ is vaguely lower semicontinuous on
$\mathfrak M^+(\mathrm Y)$.\end{lemma}

\begin{lemma}\label{integral} Fix $\mu\in\mathfrak M(\mathbf{A})$ and $\psi\in\mathrm\Phi(\mathrm X)$.
If $\int\psi\,d
R\mu$ is well defined, then
\begin{equation}
\int\psi\,dR\mu=\sum_{i\in
I}\,\alpha_i\int\psi\,d\mu^i,\label{lemma11}
\end{equation}
and $\int\psi\,d R\mu$ is finite if and only if the series on the
right converges absolutely.\end{lemma}

\begin{proof} We can certainly assume $\psi$ to be nonnegative,
for if not, we replace $\psi$ by a function~$\psi'\geqslant0$
obtained by adding to~$\psi$ a suitable constant~$c>0$, which is
always possible since a lower semicontinuous function is bounded
from below on a compact space. Hence,
$$\int\psi\,d R\mu^+\geqslant\sum_{i\in I^+, \ i\leqslant N}\,\int\psi\,d\mu^i
\quad\mbox{for all \ }N\in\mathbb N.$$ On the other hand, the sum of
$\mu^i$ over all $i\in I^+$ that do not exceed~$N$
approaches~$R\mu^+$ vaguely as $N\to\infty$; consequently, by
Lemma~\ref{lemma:lower},
\begin{equation*}
\int\psi\,d R\mu^+\leqslant\lim_{N\to\infty}\,\sum_{i\in I^+, \
i\leqslant N}\,\int\psi\,d\mu^i.\end{equation*} Combining the last
two inequalities and then letting $N\to\infty$ yields
$$\int\psi\,d R\mu^+=\sum_{i\in I^+}\,\int\psi\,d\mu^i.$$ Since the
same holds true for $R\mu^-$ and~$I^-$ instead of~$R\mu^+$
and~$I^+$, respectively, the lemma follows.\end{proof}

\begin{corollary}\label{pot.ener} Given $\mu,\,\mu_1\in\mathfrak
M(\mathbf{A})$ and $x\in\mathrm X$, we have
\begin{align}
\kappa(x,R\mu)&=\sum_{i\in
I}\,\alpha_i\kappa(x,\mu^i),\label{poten}\\
\kappa(R\mu,R\mu_1)&=\sum_{i,j\in
I}\,\alpha_i\alpha_j\kappa(\mu^i,\mu_1^j),\label{mutual}
\end{align}
each of the identities being understood in the sense that its
right-hand side is well defined whenever so is the left-hand one and
then they coincide. Furthermore, the left-hand side
in~{\rm(\ref{poten})} or in~{\rm(\ref{mutual})} is finite if and
only if the corresponding series on the right converges
absolutely.\end{corollary}

\begin{proof} Relation (\ref{poten}) is a direct consequence
of~(\ref{lemma11}), while (\ref{mutual}) follows from Fubini's
theorem (cf.~\cite[\S~8, Th.~1]{B3}) and Lemma~\ref{integral} on
account of the fact that $\kappa(x,\nu)$, where $\nu\in\mathfrak
M^+$ is given, is lower semicontinuous on~$\mathrm X$ (see,
e.\,g.,~\cite{F1}).\end{proof}

\begin{definition} Given $\mu,\,\mu_1\in\mathfrak M(\mathbf{A})$, we
call $\kappa(\,\cdot\,,\mu):=\kappa(\,\cdot\,,R\mu)$ the {\it
potential\/} of~$\mu$ and $\kappa(\mu,\mu_1):=\kappa(R\mu,R\mu_1)$
the {\it mutual energy\/} of~$\mu$ and~$\mu_1$. For $\mu\equiv\mu_1$
we get the {\it energy\/} $\kappa(\mu,\mu)$ of~$\mu$; i.\,e., if
$\kappa(R\mu,R\mu)$ is well defined, then
\begin{equation}
\kappa(\mu,\mu):=\kappa(R\mu,R\mu)=\sum_{i,j\in
I}\,\alpha_i\alpha_j\kappa(\mu^i,\mu^j).\label{ener}\end{equation}
\end{definition}

\begin{corollary}\label{finite}  For $\mu\in\mathfrak M(\mathbf{A})$ to
be of finite energy, it is necessary and sufficient that
$\mu^i\in\mathcal E$ for all $i\in I$ and
\[\sum_{i\in I}\,\|\mu^i\|^2<\infty.\]\end{corollary}

\begin{proof} This follows from (\ref{ener}) due to the inequality
$2\kappa(\nu_1,\nu_2)\leqslant\|\nu_1\|^2+\|\nu_2\|^2$, where
$\nu_1,\,\nu_2\in\mathcal E$.\end{proof}

\begin{remark} Observe that, for every $\mu\in\mathfrak M(\mathbf{A})$,
the series in~(\ref{ener}) actually defines the energy of the
corresponding infinitely dimensional vector measure~$(\mu^i)_{i\in
I}$ relative to the interaction matrix~$(\alpha_i\alpha_j)_{i,j\in
I}$\,; compare with~\cite{GR0,GR} and \cite[Chap.~5, \S~4]{NS}. Our
approach, however, is essentially based on the fact that, due to the
specific form of interaction matrix, that value can also be obtained
as the energy of the corresponding scalar Radon
measure~$R\mu$.\end{remark}

\begin{remark} Since we make no difference between $\mu\in\mathfrak
M(\mathbf{A})$ and $R\mu$ when dealing with their energies or
potentials, we shall sometimes call a measure associated with
$\mathbf{A}$ simply a {\it measure\/}~--- certainly, if this causes
no confusion.\end{remark}

\subsection{Strong topology on $\mathcal E(\mathbf{A})$}

Let $\mathcal E(\mathbf{A})$ consist of all $\mu\in\mathfrak
M(\mathbf{A})$ of finite energy $\kappa(\mu,\mu)$. Since $\mathfrak
M(\mathbf{A})$ forms a convex cone, it follows from
Corollary~\ref{finite} that so does $\mathcal E(\mathbf{A})$.

Let us treat $\mathcal E(\mathbf{A})$ as a {\it semimetric space\/}
with the semimetric
\begin{equation}
\|\mu_1-\mu_2\|:=\|\mu_1-\mu_2\|_{\mathcal
E(\mathbf{A}}:=\|R\mu_1-R\mu_2\|_\mathcal
E,\quad\mu_1,\,\mu_2\in\mathcal E(\mathbf{A});\label{seminorm}
\end{equation}
then $\mathcal E(\mathbf{A})$ and its $R$-image become {\it
isometric}. Similarly with the terminology in~$\mathcal E$, the
topology on $\mathcal E(\mathbf{A})$ defined by means of the
semimetric~(\ref{seminorm}) is called {\it strong\/}.

Two elements of $\mathcal E(\mathbf{A})$, $\mu_1$ and~$\mu_2$, are
called {\it equivalent in\/}~$\mathcal E(\mathbf{A})$ if
$\|\mu_1-\mu_2\|=0$. If, in addition, the kernel~$\kappa$ is assumed
to be strictly positive definite, then the equivalence in~$\mathcal
E(\mathbf{A})$ implies that in~$\mathfrak M(\mathbf{A})$ (namely,
then $\mu_1\cong\mu_2$), and it implies the identity (i.\,e., then
$\mu_1\equiv\mu_2$) if, moreover, all $A_i$, $i\in I$, are mutually
disjoint.

\section{Interior capacities of condensers. Elementary properties}\label{sec:4}

Given a set $\mathcal H$ in the semimetric space~$\mathcal
E(\mathbf{A})$, let us introduce the quantity
$$
\|\mathcal H\|^2:=\inf_{\nu\in\mathcal H}\,\|\nu\|^2,
$$
interpreted as $+\infty$ if $\mathcal H$ is empty. If $\|\mathcal
H\|^2<\infty$, one can consider the variational problem on the
existence of $\lambda_\mathcal H\in\mathcal H$ with minimal energy
$\|\lambda_\mathcal H\|^2=\|\mathcal H\|^2$; such a problem will be
referred to as the $\mathcal H$-{\it problem\/}. The $\mathcal
H$-problem is called {\it solvable\/} if a
minimizer~$\lambda_\mathcal H$ exists.

\subsection{Capacity of a condenser}Fix a vector-valued function~$\mathbf{g}=(g_i)_{i\in I}$,
where all $g_i:\mathbf X\to(0,\infty)$ are continuous, and a
numerical vector $\mathbf{a}=(a_i)_{i\in I}$ with $a_i>0$. Given a
condenser~$\mathbf{A}$, write
\begin{equation*}\mathfrak M^+(A_i,a_i,g_i):=\Bigl\{\nu\in\mathfrak
M^+(A_i): \ \int g_i\,d\nu=a_i\Bigr\},\quad i\in I,\end{equation*}
and let $\mathfrak M(\mathbf{A},\mathbf{a},\mathbf{g})$ consist of
all $\mu\in\mathfrak M(\mathbf{A})$ with $\mu^i\in\mathfrak
M^+(A_i,a_i,g_i)$ for all $i\in I$. Given a kernel~$\kappa$, also
write \[\mathcal E^+(A_i,a_i,g_i):=\mathfrak
M^+(A_i,a_i,g_i)\cap\mathcal E,\qquad\mathcal
E(\mathbf{A},\mathbf{a},\mathbf{g}):=\mathfrak
M(\mathbf{A},\mathbf{a},\mathbf{g})\cap\mathcal E(\mathbf{A}).\]

\begin{definition} We shall call the value
\begin{equation*}
{\rm cap}\,\mathbf A:={\rm
cap}\,(\mathbf{A},\mathbf{a},\mathbf{g}):=\frac{1}{\|\mathcal
E(\mathbf{A},\mathbf{a},\mathbf{g})\|^2} \label{def}
\end{equation*}
the ({\it interior}) {\it capacity\/} of an $(I^+,I^-)$-condenser
$\mathbf A$ (with respect to~$\kappa$, $\mathbf{a}$,
and~$\mathbf{g}$).\end{definition}

Here and in the sequel, we adopt the convention that $1/0=+\infty$.
Then, by the positive definiteness of the kernel, $0\leqslant{\rm
cap}\,\mathbf{A}\leqslant\infty$; necessary and (or) sufficient
conditions for $0<{\rm cap}\,\mathbf{A}<\infty$ to hold will be
provided in~Sec.~\ref{sec:nondeg} below (see
also~Corollary~\ref{cor:j}).

\subsection{Continuity property}
On $\mathfrak C=\mathfrak C(I^+,I^-)$, it is natural to introduce an
ordering relation~$\prec$ by declaring $\mathbf{A}'\prec\mathbf{A}$
to mean that $A_i'\subset A_i$ for all $i\in I$. Here,
$\mathbf{A}'=(A_i')_{i\in I}$. Then ${\rm
cap}\,(\,\cdot\,,\mathbf{a},\mathbf{g})$ is a nondecreasing function
of a condenser, namely
\begin{equation}
{\rm cap}\,(\mathbf{A}',\mathbf{a},\mathbf{g})\leqslant{\rm
cap}\,(\mathbf{A},\mathbf{a},\mathbf{g})\quad\mbox{whenever \
}\mathbf{A}'\prec\mathbf{A}. \label{increas'}
\end{equation}
Given $\mathbf{A}\in\mathfrak C$, let us consider the increasing
family $\{\mathbf{K}\}_{\mathbf A}$ of all compact condensers
$\mathbf{K}=(K_i)_{i\in I}\in\mathfrak C$ such that
$\mathbf{K}\prec\mathbf{A}$.

\begin{lemma}\label{lemma.cont} If $\mathbf K$ ranges over
$\{\mathbf{K}\}_{\mathbf A}$, then
\begin{equation}
{\rm
cap}\,(\mathbf{A},\mathbf{a},\mathbf{g})=\lim_{\mathbf{K}\uparrow\mathbf{A}}\,
{\rm cap}\,(\mathbf{K},\mathbf{a},\mathbf{g}).\label{cont}
\end{equation}\end{lemma}

\begin{proof} We can certainly assume ${\rm cap}\,(\mathbf{A},\mathbf{a},\mathbf{g})$ to be nonzero,
since otherwise (\ref{cont})~follows at once from~(\ref{increas'}).
Then the set $\mathcal E(\mathbf{A},\mathbf{a},\mathbf{g})$ must be
nonempty; fix~$\mu$, one of its elements. Given
$\mathbf{K}\in\{\mathbf{K}\}_{\mathbf{A}}$ and $i\in I$, let
$\mu^i_{\mathbf{K}}$ denote the trace of~$\mu^i$ upon~$K_i$, i.\,e.,
$\mu^i_{\mathbf{K}}:=\mu_{K_i}^i$. Applying Lemma~1.2.2
from~\cite{F1}, we conclude that
\begin{align}
\int g_i\,d\mu^i&=\lim_{\mathbf{K}\uparrow\mathbf{A}}\,\int
g_i\,d\mu_{\mathbf{K}}^i,\qquad i\in I,\label{w}\\
\kappa(\mu^i,\mu^j)&=\lim_{\mathbf{K}\uparrow\mathbf{A}}\,\kappa(\mu_{\mathbf{K}}^i,\mu_{\mathbf{K}}^j),\qquad
i,j\in I.\label{ww}
\end{align}

Fix $\varepsilon>0$. It follows from~(\ref{w}) and~(\ref{ww}) that
for every $i\in I$ one can choose a compact set $K_i^0\subset A_i$
so that
\begin{equation}
\frac{a_i}{\int
g_i\,d\mu^i_{K_i^0}}<1+\varepsilon\,i^{-2},\label{unif2}
\end{equation}
\begin{equation}
\bigl|\|\mu^i\|^2-\|\mu^i_{K_i^0}\|^2\bigr|<\varepsilon^2i^{-4}.\label{unif1}
\end{equation}
Having denoted $\mathbf{K}^0:=(K_i^0)_{i\in I}$, for every
$\mathbf{K}\in\{\mathbf{K}\}_{\mathbf{A}}$ that
follows~$\mathbf{K}^0$ we therefore have $\int
g_i\,d\mu_{\mathbf{K}}^i\ne0$ and
\begin{equation}\label{hatmu}
\hat{\mu}_{\mathbf{K}}:=\sum_{i\in I}\,\frac{\alpha_ia_i}{\int
g_i\,d\mu_{\mathbf{K}}^i}\,\mu_{\mathbf{K}}^i\in\mathcal
E(\mathbf{K},\mathbf{a},\mathbf{g}),\end{equation} the finiteness of
the energy being obtained from~(\ref{unif1})
and~Corollary~\ref{finite}. Thus,
\begin{equation}
\|\hat{\mu}_{\mathbf{K}}\|^2\geqslant\|\mathcal
E(\mathbf{K},\mathbf{a},\mathbf{g})\|^2.\label{www}\end{equation}

We next proceed by showing that
\begin{equation}
\|\mu\|^2=\lim_{\mathbf{K}\uparrow\mathbf{A}}\,\|\hat{\mu}_{\mathbf{K}}\|^2.\label{4w}\end{equation}
To this end, it can be assumed that $\kappa\geqslant0$; for if not,
then $\mathbf A$ must be finite since $\mathrm X$ is compact, and
(\ref{4w}) follows from~(\ref{w}) and~(\ref{ww}) when substituted
into~(\ref{ener}). Therefore, for every~$\mathbf{K}$ that
follows~$\mathbf{K}_0$ and every $i\in I$ we get
\begin{equation}
\|\mu^i_{\mathbf{K}}\|\leqslant\|\mu^i\|\leqslant\|R\mu^++R\mu^-\|,
\label{unif3}
\end{equation}
\begin{equation}
\|\mu^i-\mu^i_{\mathbf{K}}\|<\varepsilon\,i^{-2},\label{uniff}
\end{equation}
the latter being clear from (\ref{unif1}) because of
$\kappa(\mu^i_{\mathbf{K}},\mu^i-\mu^i_{\mathbf{K}})\geqslant0$.
Furthermore, by~(\ref{ener}),
\begin{equation*}
\begin{split}
\bigl|&\|\mu\|^2-\|\hat{\mu}_{\mathbf{K}}\|^2\bigr|\leqslant\sum_{i,j\in
I}\,\Bigl|\kappa(\mu^i,\mu^j)-\frac{a_i}{\int
g_i\,d\mu^i_{\mathbf{K}}}\frac{a_j}{\int
g_j\,d\mu^j_{\mathbf{K}}}\,\kappa(\mu_{\mathbf{K}}^i,\mu_{\mathbf{K}}^j)\Bigr|\\
&{}\leqslant\sum_{i,j\in
I}\,\Bigl[\kappa(\mu^i-\mu^i_{\mathbf{K}},\mu^j)+\kappa(\mu^i_{\mathbf{K}},\mu^j-\mu^j_{\mathbf{K}})+
\Bigl(\frac{a_i}{\int g_i\,d\mu^i_{\mathbf{K}}}\frac{a_j}{\int
g_j\,d\mu^j_{\mathbf{K}}}-1\Bigr)\,\kappa(\mu^i_{\mathbf{K}},\mu^j_{\mathbf{K}})\Bigr].
\end{split}
\end{equation*}
When combined with (\ref{unif2}), (\ref{unif3}), and~(\ref{uniff}),
this yields
\[
\bigl|\|\mu\|^2-\|\hat{\mu}_{\mathbf{K}}\|^2\bigr|\leqslant
M\varepsilon\quad\mbox{for all \ }\mathbf{K}\succ\mathbf{K}_0,\]
where $M$ is finite and independent of~$\mathbf K$, and the required
relation~(\ref{4w}) follows.

Substituting (\ref{www}) into~(\ref{4w}), in view of the arbitrary
choice of $\mu\in\mathcal E(\mathbf{A},\mathbf{a},\mathbf{g})$ we
get
\[
\|\mathcal
E(\mathbf{A},\mathbf{a},\mathbf{g})\|^2\geqslant\lim_{\mathbf{K}\uparrow\mathbf{A}}\,\|\mathcal
E(\mathbf{K},\mathbf{a},\mathbf{g})\|^2.
\]
Since the converse inequality is obvious from~(\ref{increas'}), the
proof is complete.\end{proof}

Let $\mathcal E_0(\mathbf{A},\mathbf{a},\mathbf{g})$ denote the
class of all $\mu\in\mathcal E(\mathbf{A},\mathbf{a},\mathbf{g})$
such that, for every $i\in I$, the support~$S(\mu^i)$ of~$\mu^i$ is
compact and contained in~$A_i$.

\begin{corollary} \label{cor:compact} The capacity ${\rm cap}\,(\mathbf{A},\mathbf{a},\mathbf{g})$
remains unchanged if the class $\mathcal
E(\mathbf{A},\mathbf{a},\mathbf{g})$ in its definition is replaced
by $\mathcal E_0(\mathbf{A},\mathbf{a},\mathbf{g})$.  In other
words,
\[\|\mathcal E(\mathbf{A},\mathbf{a},\mathbf{g})\|^2=
\|\mathcal E_0(\mathbf{A},\mathbf{a},\mathbf{g})\|^2.\]\end{corollary}

\begin{proof} We can assume that $\|\mathcal E(\mathbf{A},\mathbf{a},\mathbf{g})\|^2<\infty$, since
otherwise the corollary follows from $\mathcal
E_0(\mathbf{A},\mathbf{a},\mathbf{g})\subset\mathcal
E(\mathbf{A},\mathbf{a},\mathbf{g})$. Then, by~(\ref{increas'})
and~(\ref{cont}), for every $\varepsilon>0$ there exists a compact
condenser $\mathbf{K}\prec\mathbf{A}$ such that $\|\mathcal
E(\mathbf{K},\mathbf{a},\mathbf{g})\|^2\leqslant\|\mathcal
E(\mathbf{A},\mathbf{a},\mathbf{g})\|^2+\varepsilon$. This proves
the corollary when combined with the inequalities \[\|\mathcal
E(\mathbf{K},\mathbf{a},\mathbf{g})\|^2\geqslant\|\mathcal
E_0(\mathbf{A},\mathbf{a},\mathbf{g})\|^2\geqslant\|\mathcal
E(\mathbf{A},\mathbf{a},\mathbf{g})\|^2.\]\end{proof}

\subsection{When does $0<{\rm cap}\,\mathbf{A}<\infty$ hold?}\label{sec:nondeg}
In all that follows it is required that
\begin{equation}
{\rm cap}\,(\mathbf{A},\mathbf{a},\mathbf{g})>0.\label{nonzero1}
\end{equation}

\begin{lemma}\label{lemma:nonzero} For {\rm(\ref{nonzero1})} to hold, it is
necessary and sufficient that any of the following three conditions
be satisfied:
\begin{itemize}
\item[\rm(i)] $\mathcal E(\mathbf{A},\mathbf{a},\mathbf{g})$ is nonempty;
\item[\rm(ii)] there exist $\nu_i\in\mathcal E^+(A_i,a_i,g_i)$ for all
$i\in I$ such that $\sum_{i\in I}\,\|\nu_i\|^2<\infty$;
\item[\rm(iii)] $\sum_{i\in I}\,\|\mathcal
E^+(A_i,a_i,g_i)\|^2<\infty$.
\end{itemize}
\end{lemma}

\begin{proof} The equivalency of (\ref{nonzero1}) and
(i) is obvious, while that of~(i) and~(ii) can be obtained directly
from Corollary~\ref{finite}. If (iii)~is true, then for every $i\in
I$ one can choose $\nu_i\in\mathcal E^+(A_i,a_i,g_i)$ so that
$\|\nu_i\|^2<\|\mathcal E^+(A_i,a_i,g_i)\|^2+i^{-2}$, and (ii)
follows. Since (ii)~obviously yields~(iii), the proof is
complete.\end{proof}

Let $C(\,\cdot\,)$ denote the {\it interior capacity\/} of a set
with respect to the kernel~$\kappa$~\cite{F1}.

\begin{corollary}\label{cor:nonzero} For {\rm(\ref{nonzero1})} to be
satisfied, it is necessary that
\begin{equation}
C(A_i)>0\quad\mbox{for all \ } i\in I. \label{cap2}
\end{equation}
If $\mathbf{A}$ is finite, then {\rm(\ref{nonzero1})} and
{\rm(\ref{cap2})} are actually equivalent\,\footnote{However,
(\ref{nonzero1}) and (\ref{cap2}) are no longer equivalent if
$\mathbf{A}$ is infinite~---
cf.~Corollary~\ref{lemma:nonzero.infinite}.}.\end{corollary}

\begin{proof} For Lemma~\ref{lemma:nonzero}, (ii) to hold, it is necessary that,
for every $i\in I$, there exists a nonzero nonnegative measure of
finite energy, compactly supported in~$A_i$, which in turn is
equivalent to~(\ref{cap2}) according to~\cite[Lemma~2.3.1]{F1}.
Since the former implication can obviously be inverted if $\mathbf
A$ is finite, the proof is complete.\end{proof}

Let $g_{i,\inf}$ and $g_{i,\sup}$ be the infimum and the supremum
of~$g_i$ over~$A_i$. Also write
\[\mathbf{g}_{\inf}:=\inf_{i\in I}\,g_{i,\inf},\qquad\mathbf{g}_{\sup}:=\sup_{i\in
I}\,g_{i,\sup}.\]

\begin{corollary}\label{lemma:nonzero.infinite} Assume $0<\mathbf{g}_{\inf}\leqslant
\mathbf{g}_{\sup}<\infty$. Then {\rm(\ref{nonzero1})} holds if and
only if
\begin{equation*}\sum_{i\in I}\,\frac{a_i^2}{C(A_i)}<\infty.\end{equation*}\end{corollary}

\begin{proof} Lemma~\ref{lemma:nonzero}, (iii) implies the corollary when combined
with the inequalities
\begin{equation}
\frac{a^2_i}{g_{i,\sup}^2\,C(A_i)}\leqslant\|\mathcal
E^+(A_i,a_i,g_i)\|^2\leqslant\frac{a^2_i}{g_{i,\inf}^2\,C(A_i)}\,,\quad
i\in I,\label{5w}
\end{equation}
to be proved below by reasons of homogeneity.

To establish (\ref{5w}), fix $i\in I$. One can certainly assume
$C(A_i)$ to be nonzero, for otherwise Corollary~\ref{cor:nonzero}
with $I=\{i\}$ shows that each of the three parts in~(\ref{5w})
equals~$+\infty$. Consequently, there exists $\theta_i\in\mathcal
E^+(A_i,1,1)$. Since
$$\hat{\theta}_i:=\frac{a_i\theta_i}{\int g_i\,d\theta_i}\in\mathcal
E^+(A_i,a_i,g_i),$$ we get $$a_i^2\,\|\theta_i\|^2\geqslant
g_{i,\inf}^2\,\|\hat{\theta}_i\|^2\geqslant g_{i,\inf}^2\,\|\mathcal
E^+(A_i,a_i,g_i)\|^2,$$ and the right-hand side in~(\ref{5w}) is
obtained by letting $\theta_i$ range over $\mathcal E^+(A_i,1,1)$.

To verify the left-hand side, fix $\omega_i\in\mathcal
E^+(A_i,a_i,g_i)$. Then
$$0<a_i\,g_{i,\sup}^{-1}\leqslant\omega_i(\mathrm X)\leqslant
a_i\,g_{i,\inf}^{-1}<\infty.$$ Hence, $\omega_i/\omega_i(\mathrm
X)\in\mathcal E^+(A_i,1,1)$ and
$$\|\omega_i\|^2\geqslant{a_i^2}\,{g_{i,\sup}^{-2}}\,\|\mathcal E^+(A_i,1,1)\|^2.$$
In view of the arbitrary choice of $\omega_i\in\mathcal
E^+(A_i,a_i,g_i)$, this completes the proof.\end{proof}

In the following assertion, providing necessary conditions for ${\rm
cap}\,\mathbf{A}$ to be finite, it is assumed that $g_{i,\inf}>0$
for all $i\in I$.

\begin{lemma}\label{lemma:finite} If ${\rm cap}\,(\mathbf{A},\mathbf{a},\mathbf{g})<\infty$,
then there exists $j\in I$ with $C(A_j)<\infty$.
\end{lemma}

\begin{proof}Under the assumptions of the lemma, suppose $C(A_i)=\infty$ for all $i\in I$.
Given $\varepsilon>0$,
for every~$i\in I$ one can choose $\mu^i\in\mathcal E^+(A_i,1,1)$
with compact support so that $\|\mu^i\|\leqslant\varepsilon
a_i^{-1}i^{-2}\,g_{i,\inf}$. Since then \[ \hat{\mu}:=\sum_{i\in
I}\,\frac{\alpha_i a_i\mu^i}{\int g_i\,d\mu^i}\in\mathcal
E(\mathbf{A},\mathbf{a},\mathbf{g})\] and
$\|\hat{\mu}\|\leqslant\varepsilon\sum_{i\in I}\,i^{-2}$, we arrive
at a contradiction by letting $\varepsilon\to0$.\end{proof}

\begin{remark}\label{ns}It will be shown by Corollary~\ref{cor:j} below that,
under certain additional restrictions, Lemma~\ref{lemma:finite} can
be inverted.\end{remark}

\section{Minimizing measures on condensers: existence, uniqueness,
$\mathbf{A}$-vague compactness, continuity}\label{sec:5}

Because of~(\ref{nonzero1}), we are naturally led to the $\mathcal
E(\mathbf{A},\mathbf{a},\mathbf{g})$-{\it problem\/}
(cf.~Sec.~\ref{sec:4}), i.\,e., the problem on the existence of
$\lambda=\lambda_{\mathbf{A}}\in\mathcal
E(\mathbf{A},\mathbf{a},\mathbf{g})$ with minimal energy
\[\|\lambda_{\mathbf{A}}\|^2=\|\mathcal E(\mathbf{A},\mathbf{a},\mathbf{g})\|^2.\] Let
$\mathfrak S(\mathbf{A},\mathbf{a},\mathbf{g})$ denote the class of
all minimizers~$\lambda_{\mathbf{A}}$.

\subsection{Uniqueness properties of $\lambda_{\mathbf{A}}$} We start by observing that any two
minimizers (if exist) are {\it equivalent in\/}~$\mathcal
E(\mathbf{A})$, i.\,e.,
\begin{equation*}\|\lambda_1-\lambda_2\|=0\quad\mbox{for all \
}\lambda_1,\,\lambda_2\in\mathfrak
S(\mathbf{A},\mathbf{a},\mathbf{g}).\label{uniqueness}\end{equation*}
Indeed, this follows from the convexity of $\mathcal
E(\mathbf{A},\mathbf{a},\mathbf{g})$ and the parallelogram identity
in~$\mathcal E$, applied to~$R\lambda_1$ and~$R\lambda_2$. Thus,
$\lambda_1\cong\lambda_2$ provided the kernel~$\kappa$ is strictly
positive definite, and $\lambda_1\equiv\lambda_2$ if, moreover, all
$A_i$, $i\in I$, are mutually disjoint.

What about the existence of minimizers? Assume for a moment that
$\mathbf A$~is {\it finite} and {\it compact} and that $\kappa$ is
{\it continuous on\/} $A^+\times A^-$. Then $\mathfrak
M(\mathbf{A},\mathbf{a},\mathbf{g})$ is $\mathbf{A}$-va\-gue\-ly
compact while $\|\mu\|^2$ is $\mathbf{A}$-va\-gue\-ly lower
semicontinuous on~$\mathcal E(\mathbf{A})$ and, therefore, the
solvability of the $\mathcal
E(\mathbf{A},\mathbf{a},\mathbf{g})$-problem follows.
See~\cite[Th.~2.30]{O}; cf.~also~\cite{GR0,GR,NS,ST}, related to the
logarithmic kernel in the plane.

However, these arguments break down if any of the above three
assumptions is dropped. In particular, $\mathfrak
M(\mathbf{A},\mathbf{a},\mathbf{g})$ is no longer
$\mathbf{A}$-vaguely compact if $\mathbf{A}$ is noncompact.

To solve the problem on the existence of
minimizers~$\lambda_{\mathbf{A}}$ in the general case where a
condenser $\mathbf{A}$ is infinite and (or) noncompact, we develop
an approach based on both the $\mathbf{A}$-vague and strong
topologies in the semimetric space~$\mathcal E(\mathbf{A})$,
introduced for finite condensers in~\cite{Z3,Z4,Z6}.

\subsection{Standing assumptions}\label{sec:standing}

Unless explicitly stated otherwise, in all that follows it is
required that the kernel~$\kappa$ is consistent and either
$I^-=\varnothing$, or the following conditions are both satisfied:
\begin{equation}\label{s}\sum_{i\in
I}\,a_ig_{i,\inf}^{-1}<\infty,\end{equation}
\begin{equation}\sup_{x\in A^+,\ y\in
A^-}\,\kappa(x,y)<\infty.\label{bou}
\end{equation}

\begin{remark}
These assumptions on a kernel are not too restrictive. In
particular, they all are satisfied by the Newtonian, Riesz, or Green
kernels in~$\mathbb R^n$, $n\geqslant2$, provided the Euclidean
distance between $A^+$ and $A^-$ is nonzero.\end{remark}

\subsection{Existence of minimizers $\lambda_{\mathbf{A}}$. $\mathbf A$-vague compactness}
A proposition $R(x)$ involving a variable point $x\in\mathbf X$ is
said to subsist {\it nearly everywhere\/}~(n.\,e.) in~$E$, where $E$
is a given subset of~$\mathbf X$, if the set of all $x\in E$ for
which $R(x)$ fails to hold is of interior capacity zero~\cite{F1}.

\begin{theorem}\label{exist}
For every $i\in I$, assume that either $g_{i,\sup}<\infty$, or there
exist $r_i\in(1,\infty)$ and $\omega_i\in\mathcal E$ such that
\begin{equation}
g^{r_i}(x)\leqslant\kappa(x,\omega_i)\quad\mbox{n.\,e. in \ } A_i.
\label{growth}
\end{equation}
If, moreover, $A_i$ either is compact or has finite
capacity\,\footnote{Note that a compact set $K\subset\mathbf X$
might be of infinite capacity; $C(K)$ is necessarily finite provided
the kernel is strictly positive definite~\cite{F1}. On the other
hand, even for the Newtonian kernel, sets of finite capacity might
be noncompact (see~\cite{L}).}
\begin{equation*} C(A_i)<\infty, \label{capacityfinite}
\end{equation*} then for any vector $\mathbf{a}$ the class
$\mathfrak S(\mathbf{A},\mathbf{a},\mathbf{g})$ is nonempty and
$\mathbf{A}$-vaguely compact.
\end{theorem}

\begin{corollary}\label{cor:exist} If $\mathbf{A}=\mathbf{K}$ is compact,
then for any $\mathbf{a}$ and~$\mathbf{g}$ the class $\mathfrak
S(\mathbf{A},\mathbf{a},\mathbf{g})$ is nonempty and
$\mathbf{A}$-vaguely compact.
\end{corollary}

\begin{proof}This is an immediate consequence of Theorem~\ref{exist},
since $g_i$ is bounded on~$K_i$.\end{proof}

\subsection{On continuity of minimizers}When approaching $\mathbf{A}$ by the increasing
family~$\{\mathbf{K}\}_{\mathbf{A}}$ of the compact condensers
$\mathbf{K}\prec\mathbf{A}$, we shall always suppose all
those~$\mathbf{K}$ to be of capacity nonzero. This involves no loss
of generality, which is clear from~(\ref{nonzero1}) and
Lemma~\ref{lemma.cont}. Choose an arbitrary
$\lambda_{\mathbf{K}}\in\mathfrak
S(\mathbf{A},\mathbf{a},\mathbf{g})$~--- its existence has been
ensured by Corollary~\ref{cor:exist}.

\begin{theorem}\label{cor:cont}Let all
the conditions of Theorem~{\rm\ref{exist}} be satisfied. Then every
$\mathbf{A}$-vague cluster point of
$(\lambda_{\mathbf{K}})_{\mathbf{K}\in\{\mathbf{K}\}_{\mathbf{A}}}$
(such a cluster point exists) belongs to $\mathfrak
S(\mathbf{A},\mathbf{a},\mathbf{g})$. Furthermore, if
$\lambda_{\mathbf{A}}\in\mathfrak
S(\mathbf{A},\mathbf{a},\mathbf{g})$ is arbitrarily given, then
\[\lim_{\mathbf{K}\uparrow\mathbf{A}}\,\|\lambda_{\mathbf{K}}-\lambda_{\mathbf{A}}\|^2=0.\]\end{theorem}

Thus, under the assumptions of Theorem~\ref{cor:cont}, if moreover
$\kappa$ is strictly positive definite and all~$A_i$, $i\in I$, are
mutually disjoint, then the (unique)
minimizer~$\lambda_{\mathbf{K}}$ approaches the (unique)
minimizer~$\lambda_{\mathbf{A}}$ both $\mathbf A$-vaguely and
strongly as $\mathbf{K}\uparrow\mathbf{A}$.

The proofs of Theorems~\ref{exist} and~\ref{cor:cont}, to be given
in Sec.~\ref{sec:proof.th.str} (see also Sec.~\ref{sec:extremal} for
certain crucial auxiliary notions and results), are based on a
theorem on the strong completeness of proper subspaces of~$\mathcal
E(\mathbf{A})$, which is a subject of the next section.

\section{Strong completeness of measures associated with condensers}\label{sec:11}

As always, assume all the standing assumptions, stated
in~Sec.~\ref{sec:standing}, to hold. Having denoted
\begin{equation*}
\mathfrak
M(\mathbf{A},\leqslant\!\mathbf{a},\mathbf{g}):=\Bigl\{\mu\in\mathfrak
M(\mathbf{A}):\quad\int g_i\,d\mu^i\leqslant a_i\mbox{ \ for all \ }
i\in I\Bigr\},
\end{equation*}
we treat $\mathcal
E(\mathbf{A},\leqslant\!\mathbf{a},\mathbf{g}):=\mathfrak
M(\mathbf{A},\leqslant\!\mathbf{a},\mathbf{g})\cap\mathcal
E(\mathbf{A})$ as a topological subspace of the semimetric space
$\mathcal E(\mathbf{A})$; the induced topology is likewise called
the {\it strong\/} topology.

Our purpose is to show that $\mathcal
E(\mathbf{A},\leqslant\!\mathbf{a},\mathbf{g})$ is strongly
complete.

\subsection{Auxiliary assertions}

\begin{lemma}\label{lem:aux1} $\mathfrak M(\mathbf{A},\leqslant\!\mathbf{a},\mathbf{g})$ is
$\mathbf{A}$-vaguely bounded and, hence, $\mathbf{A}$-vaguely
compact.\end{lemma}

\begin{proof} Fix $i\in I$, and let a compact set $K\subset A_i$
be given. Since $g_i$ is positive and continuous, the relation
$$
a_i\geqslant\int g_i\,d\mu^i\geqslant\mu^i(K)\,\min_{x\in K}\
g_i(x),\quad\mbox{where \ }\mu\in\mathfrak
M(\mathbf{A},\leqslant\!\mathbf{a},\mathbf{g}),
$$
yields \[\sup_{\mu\in\mathfrak
M(\mathbf{A},\leqslant\mathbf{a},\mathbf{g})}\,\mu^i(K)<\infty.\]
This implies that $\mathfrak
M(\mathbf{A},\leqslant\!\mathbf{a},\mathbf{g})$ is
$\mathbf{A}$-vaguely bounded, and hence it is $\mathbf{A}$-va\-guely
relatively compact by Lemma~\ref{lem:vaguecomp}. Since it is
$\mathbf{A}$-vaguely closed in consequence
of~Lemma~\ref{lemma:lower}, the desired assertion
follows.\end{proof}

\begin{lemma}\label{lem:aux2} If a net
$(\mu_s)_{s\in S}\subset\mathcal
E(\mathbf{A},\leqslant\!\mathbf{a},\mathbf{g})$ is strongly bounded,
then its $\mathbf A$-vague adherence is nonempty and contained in
$\mathcal
E(\mathbf{A},\leqslant\!\mathbf{a},\mathbf{g})$.\end{lemma}

\begin{proof} According to Lemma~\ref{lem:aux1}, the $\mathbf{A}$-vague adherence of
$(\mu_s)_{s\in S}$ is nonempty and contained in~$\mathfrak
M(\mathbf{A},\leqslant\!\mathbf{a},\mathbf{g})$. To establish the
lemma, it is enough to prove that every its element~$\mu$ is of
finite energy.

To this end, observe that $(R\mu_s)_{s\in S}$ is strongly bounded
by~(\ref{ener}). We proceed by showing that so are the nets
$(R\mu_s^+)_{s\in S}$ and $(R\mu_s^-)_{s\in S}$, i.\,e.,
\begin{equation}
\sup_{s\in S}\,\|R\mu_s^\pm\|^2<\infty.  \label{7.1}
\end{equation}
Of course, this needs to be proved only when $I^-\ne\varnothing$;
then, according to the standing assumptions, (\ref{s})
and~(\ref{bou}) both hold. Since $\int g_i\,d\mu_s^i\leqslant a_i$,
we get
\begin{equation}
\sup_{s\in S}\,\mu_s^i(\mathbf X)\leqslant
a_ig_{i,\inf}^{-1}\quad\mbox{for all \ }i\in I. \label{7.3}
\end{equation}
Consequently, by (\ref{s}), \[\sup_{s\in S}\,R\mu_s^\pm(\mathbf
X)\leqslant \sum_{i\in I}\,a_ig_{i,\inf}^{-1}<\infty.\] Because
of~(\ref{bou}), this implies that $\kappa(R\mu^+_s,R\mu^-_s)$
remains bounded from above on~$S$; hence, so do $\|R\mu^+_s\|^2$ and
$\|R\mu^-_s\|^2$.

Now, if $(\mu_d)_{d\in D}$ is a subnet of $(\mu_s)_{s\in S}$ that
converges $\mathbf{A}$-vaguely to~$\mu$, then, by
Lemma~\ref{lem:vague}, $(R\mu^+_d)_{d\in D}$ and $(R\mu^-_d)_{d\in
D}$ converge vaguely to~$R\mu^+$ and~$R\mu^-$, respectively.
Therefore application of Lemma~\ref{lemma:lower} with $\mathrm
Y=\mathrm X\times\mathrm X$ and $\psi=\kappa$ enables us to conclude
from~(\ref{7.1}) that $R\mu^+$ and $R\mu^-$ are of finite energy.
This yields $\kappa(\mu,\mu)<\infty$ as required.\end{proof}

\begin{corollary}\label{cor:aux1} If $(\mu_s)_{s\in
S}\subset\mathcal E(\mathbf{A},\leqslant\!\mathbf{a},\mathbf{g})$ is
strongly bounded, then
\begin{equation} \sup_{s\in S}\,\|\mu_s^i\|^2<\infty,\quad i\in I.\label{7.1i}
\end{equation}
\end{corollary}

\begin{proof} It is seen from~(\ref{7.1}) that the claimed relation
(\ref{7.1i}) will be proved once we show that
\begin{equation}\sum_{i,j\in
I^\pm}\,\kappa(\mu_s^i,\mu_s^j)\geqslant
C>-\infty,\label{bdbl}\end{equation} where $C$ is independent
of~$s$. Since this is obvious when $\kappa\geqslant0$, one can
assume $\mathrm X$ to be compact. Then $\kappa$, being lower
semicontinuous, is bounded from below on~$\mathrm X$ (say by~$-c$,
where $c>0$), while $\mathbf A$ is finite. Furthermore, then
$g_{i,\inf}>0$ for all $i\in I$ and, hence, (\ref{7.3}) holds. This
implies that
$\kappa(\mu_s^i,\mu_s^j)\geqslant-a_ia_j\,g_{i,\inf}^{-1}\,g_{j,\inf}^{-1}\,c$
for all $i,\,j\in I$, and (\ref{bdbl}) follows.\end{proof}

\subsection{Strong completeness of $\mathcal E(\mathbf{A},\leqslant\!\mathbf{a},\mathbf{g})$}\label{sec:strong}

\begin{theorem}\label{th:strong} The
semimetric space $\mathcal
E(\mathbf{A},\leqslant\!\mathbf{a},\mathbf{g})$ is complete. In more
detail, if $(\mu_s)_{s\in S}$ is a strong Cauchy net in $\mathcal
E(\mathbf{A},\leqslant\!\mathbf{a},\mathbf{g})$ and $\mu$~is one of
its $\mathbf A$-vague cluster point (such a $\mu$ exists), then
$\mu\in\mathcal E(\mathbf{A},\leqslant\!\mathbf{a},\mathbf{g})$ and
\begin{equation}
\lim_{s\in S}\,\|\mu_s-\mu\|^2=0.\label{str}
\end{equation}
Assume, in addition, that the kernel is strictly positive definite
and all $A_i$, $i\in I$, are mutually disjoint. If moreover
$(\mu_s)_{s\in S}\subset\mathcal
E(\mathbf{A},\leqslant\!\mathbf{a},\mathbf{g})$ converges strongly
to $\mu_0\in\mathcal E(\mathbf A)$, then actually $\mu_0\in\mathcal
E(\mathbf{A},\leqslant\!\mathbf{a},\mathbf{g})$ and $\mu_s\to\mu_0$
$\mathbf A$-vaguely.\end{theorem}

\begin{proof} Fix a
strong Cauchy net  $(\mu_s)_{s\in S}\subset\mathcal
E(\mathbf{A},\leqslant\!\mathbf{a},\mathbf{g})$. Since such a net
converges strongly to every its strong cluster point, $(\mu_s)_{s\in
S}$ can certainly be assumed to be strongly bounded. Then, by
Lemma~\ref{lem:aux2}, there exists an $\mathbf A$-vague cluster
point~$\mu$ of~$(\mu_s)_{s\in S}$, and
\begin{equation}
\mu\in\mathcal E(\mathbf{A},\leqslant\!\mathbf{a},\mathbf{g}).
\label{leqslant1}
\end{equation}
We next proceed by verifying (\ref{str}). Of course, there is no
loss of generality in assuming $(\mu_s)_{s\in S}$ to converge
$\mathbf A$-vaguely to~$\mu$. Then, by Lemma~\ref{lem:vague},
$(R\mu^+_s)_{s\in S}$ and $(R\mu^-_s)_{s\in S}$ converge vaguely
to~$R\mu^+$ and~$R\mu^-$, respectively. Since, by~(\ref{7.1}), these
nets are strongly bounded in~$\mathcal E^+$, the property~(C$_2$)
(see~Sec.~\ref{sec:2}) shows that they approach~$R\mu^+$
and~$R\mu^-$, respectively, in the weak topology as well, and so
$R\mu_s\to R\mu$ weakly. This gives, by~(\ref{seminorm}),
$$
\|\mu_s-\mu\|^2=\|R\mu_s-R\mu\|^2=\lim_{l\in
S}\,\kappa(R\mu_s-R\mu,R\mu_s-R\mu_l),
$$
and hence, by the Cauchy-Schwarz inequality,
$$\|\mu_s-\mu\|^2\leqslant
\|\mu_s-\mu\|\,\liminf_{l\in S}\,\|\mu_s-\mu_l\|,
$$
which proves (\ref{str}) as required, because $\|\mu_s-\mu_l\|$
becomes arbitrarily small when $s,\,l\in S$ are both sufficiently
large.

Suppose now that $\kappa$ is strictly positive definite, while all
$A_i$, $i\in I$, are mutually disjoint, and let the net
$(\mu_s)_{s\in S}$ converge strongly to some $\mu_0\in\mathcal
E(\mathbf A)$. Given an $\mathbf A$-vague limit point~$\mu$
of~$(\mu_s)_{s\in S}$, we conclude from~(\ref{str}) that
$\|\mu_0-\mu\|=0$, hence $\mu_0\cong\mu$ since $\kappa$ is strictly
positive definite, and finally $\mu_0\equiv\mu$ because $A_i$, $i\in
I$, are mutually disjoint. In view of~(\ref{leqslant1}), this means
that $\mu_0\in\mathcal
E(\mathbf{A},\leqslant\!\mathbf{a},\mathbf{g})$, which is a part of
the desired conclusion. Moreover, $\mu_0$ has thus been shown to be
identical to any $\mathbf A$-vague cluster point of~$(\mu_s)_{s\in
S}$. Since the $\mathbf A$-vague topology is Hausdorff, this implies
that $\mu_0$ is actually the $\mathbf A$-vague limit
of~$(\mu_s)_{s\in S}$ (cf.~\cite[Chap.~I, \S~9, n$^\circ$\,1,
cor.]{B1}), which completes the proof.\end{proof}

\begin{remark} In view of the fact that
the semimetric space $\mathcal
E(\mathbf{A},\leqslant\!\mathbf{a},\mathbf{g})$ is isometric to its
$R$-image, Theorem~\ref{th:strong} has thus singled out a {\it
strongly complete\/} topological subspace of the pre-Hilbert
space~$\mathcal E$, whose elements are {\it signed\/} measures. This
is of independent interest since, according to a well-known
counterexample by~H.~Cartan~\cite{Car}, all the space~$\mathcal E$
is strongly incomplete even for the Newtonian kernel $|x-y|^{2-n}$
in~$\mathbb R^n$, $n\geqslant3$.\end{remark}

\begin{remark} Assume $\kappa$ is strictly positive
definite (hence, perfect). If moreover $I^-=\varnothing$, then
Theorem~\ref{th:strong} remains valid for $\mathcal E(\mathbf A)$ in
place of $\mathcal E(\mathbf{A},\leqslant\!\mathbf{a},\mathbf{g})$
(cf.~Theorem~\ref{th:1}). A question still unanswered is whether
this is the case if $I^+$ and $I^-$ are both nonempty. We can
however show that this is really so for the Riesz kernels
$|x-y|^{\alpha-n}$, $0<\alpha<n$, in~$\mathbb R^n$, $n\geqslant2$
(cf.~\cite[Th.~1]{Z2}). The proof utilizes Deny's theorem~\cite{D1}
stating that, for the Riesz kernels, $\mathcal E$~can be completed
by making use of distributions of finite energy.\end{remark}

\section{Extremal measures}\label{sec:extremal} To apply Theorem~\ref{th:strong} to the
$\mathcal E(\mathbf{A},\mathbf{a},\mathbf{g})$-problem, we next
proceed by introducing a concept of extremal measure defined as a
strong and, simultaneously, the $\mathbf A$-vague limit of a
minimizing net. Since the energy of such a measure equals the
infimum value $\|\mathcal E(\mathbf{A},\mathbf{a},\mathbf{g})\|^2$,
it serves as a minimizer if and only if it belongs to the class
$\mathcal E(\mathbf{A},\mathbf{a},\mathbf{g})$. See below for the
strict definition and related auxiliary results.

\subsection{Minimizing nets; their $\mathbf A$-vague and strong cluster points}

\begin{definition}We call a net $(\mu_s)_{s\in S}$ {\it minimizing\/} if
$(\mu_s)_{s\in S}\subset\mathcal
E_0(\mathbf{A},\mathbf{a},\mathbf{g})$ and
\begin{equation}\lim_{s\in S}\,\|\mu_s\|^2=\|\mathcal E(\mathbf{A},\mathbf{a},\mathbf{g})\|^2.\label{min}\end{equation}\end{definition}

Let $\mathbb M(\mathbf{A},\mathbf{a},\mathbf{g})$ consist of all
minimizing nets, and let $\mathcal
M(\mathbf{A},\mathbf{a},\mathbf{g})$ stand for the union of all
their $\mathbf A$-vague cluster sets. Note that $\mathbb
M(\mathbf{A},\mathbf{a},\mathbf{g})$ is nonempty, which is clear
from~(\ref{nonzero1}) in view of Corollary~\ref{cor:compact}. Hence,
according to Lemma~\ref{lem:aux2}, $\mathcal
M(\mathbf{A},\mathbf{a},\mathbf{g})$ is nonempty as well, and
\begin{equation}\label{ME} \mathcal M(\mathbf{A},\mathbf{a},\mathbf{g})\subset\mathcal
E(\mathbf{A},\leqslant\!\mathbf{a},\mathbf{g}).\end{equation}

\begin{definition}\label{def:extr} A
measure $\gamma\in\mathcal E(\mathbf A)$ is called {\it extremal\/}
in the $\mathcal E(\mathbf{A},\mathbf{a},\mathbf{g})$-problem if
there exists $(\mu_s)_{s\in S}\in\mathbb
M(\mathbf{A},\mathbf{a},\mathbf{g})$ converging to~$\gamma$ both
strongly and $\mathbf A$-vaguely; such a net $(\mu_s)_{s\in S}$ is
said to {\it generate\/}~$\gamma$. The collection of all extremal
measures will be denoted by $\mathfrak
E(\mathbf{A},\mathbf{a},\mathbf{g})$.
\end{definition}

\subsection{Extremal measures: existence, uniqueness, and compactness}It
follows from Definition~\ref{def:extr} and relations~(\ref{min})
and~(\ref{ME}) that an extremal measure~$\gamma$ (if exists) belongs
to the class $\mathcal
E(\mathbf{A},\leqslant\!\mathbf{a},\mathbf{g})$ and satisfies the
identity
\begin{equation}\label{extremal}
\|\gamma\|^2=\|\mathcal E(\mathcal A,a,g)\|^2.\end{equation}

\begin{lemma}\label{lemma:WM} The following assertions hold true:
\begin{itemize}
\item[\rm(i)] From every minimizing net one can select a subnet generating an extremal measure;
hence, $\mathfrak E(\mathbf{A},\mathbf{a},\mathbf{g})$ is nonempty.
Furthermore,
\begin{equation}\mathfrak E(\mathbf{A},\mathbf{a},\mathbf{g})=\mathcal M(\mathbf{A},\mathbf{a},\mathbf{g}).\label{WM}\end{equation}
\item[\rm(ii)] Every minimizing
net converges strongly to every extremal measure; consequently,
$\mathfrak E(\mathbf{A},\mathbf{a},\mathbf{g})$ is contained in an
equivalence class in~$\mathcal E(\mathbf A)$.
\item[\rm(iii)] The class $\mathfrak E(\mathbf{A},\mathbf{a},\mathbf{g})$ is $\mathbf A$-vaguely compact.
\end{itemize}
\end{lemma}

\begin{proof}Fix $(\mu_s)_{s\in S}$ and $(\nu_t)_{t\in T}$ in $\mathbb
M(\mathbf{A},\mathbf{a},\mathbf{g})$. It is seen with standard
arguments that
\begin{equation}
\lim_{(s,t)\in S\times T}\,\|\mu_s-\nu_t\|^2=0, \label{fund}
\end{equation}
where $S\times T$ denotes the directed product of the directed
sets~$S$ and~$T$ (see, e.\,g.,~\cite[Chap.~2,~\S~3]{K}). Indeed, by
the convexity of the class $\mathcal
E(\mathbf{A},\mathbf{a},\mathbf{g})$,
$$
2\,\|\mathcal E(\mathbf{A},\mathbf{a},\mathbf{g})\|
\leqslant{\|\mu_s+\nu_t\|}\leqslant\|\mu_s\|+\|\nu_t\|,
$$
and hence, by (\ref{min}),
$$
\lim_{(s,t)\in S\times T}\,\|\mu_s+\nu_t\|^2=4\,\|\mathcal
E(\mathbf{A},\mathbf{a},\mathbf{g})\|^2.
$$
Then the parallelogram identity in~$\mathcal E$, applied to~$R\mu_s$
and~$R\nu_t$, gives~(\ref{fund}).

Relation~(\ref{fund}) implies that the net $(\mu_s)_{s\in S}$ is
strongly fundamental. Therefore, if $\mu_0$ is one of its $\mathbf
A$-vague cluster points (such a~$\mu_0$ exists), then
$\mu_s\to\mu_0$ strongly by Theorem~\ref{th:strong}. This means that
$\mu_0$ is an extremal measure and, hence, $\mathcal
M(\mathbf{A},\mathbf{a},\mathbf{g})\subset\mathfrak
E(\mathbf{A},\mathbf{a},\mathbf{g})$. Since the inverse inclusion is
obvious, (\ref{WM}) follows.

To prove (ii), fix $(\mu_s)_{s\in S}\in\mathbb
M(\mathbf{A},\mathbf{a},\mathbf{g})$ and $\gamma\in\mathfrak
E(\mathbf{A},\mathbf{a},\mathbf{g})$. By Definition~\ref{def:extr},
one can choose a net in~$\mathbb
M(\mathbf{A},\mathbf{a},\mathbf{g})$, say $(\nu_t)_{t\in T}$,
converging to~$\gamma$ strongly. Repeated application
of~(\ref{fund}) shows that also $(\mu_s)_{s\in S}$ converges
to~$\gamma$ strongly, as claimed.

To verify (iii), it is enough to show that $\mathcal
M(\mathbf{A},\mathbf{a},\mathbf{g})$ is $\mathbf A$-vaguely compact.
Fix $(\gamma_s)_{s\in S}\subset\mathcal
M(\mathbf{A},\mathbf{a},\mathbf{g})$. It follows from~(\ref{ME})
and~Lemma~\ref{lem:aux1} that there exists an $\mathbf{A}$-vague
cluster point~$\gamma_0$ of $(\gamma_s)_{s\in S}$; let
$(\gamma_t)_{t\in T}$ be a subnet of~$(\gamma_s)_{s\in S}$ that
converges $\mathbf{A}$-vaguely to~$\gamma_0$. Then, for every $t\in
T$, there exists $(\mu_{s_t})_{s_t\in S_t}\in\mathbb
M(\mathbf{A},\mathbf{a},\mathbf{g})$ converging $\mathbf A$-vaguely
to~$\gamma_t$. Consider the Cartesian product $\prod\,\{S_t: t\in
T\}$~--- that is, the collection of all functions~$\beta$ on~$T$
with $\beta(t)\in S_t$, and let~$D$ denote the directed product
$T\times\prod\,\{S_t: t\in T\}$ (see,
e.\,g.,~\cite[Chap.~2,~\S~3]{K}). Given $(t,\beta)\in D$, write
$\mu_{(t,\beta)}:=\mu_{\beta(t)}$.  Then the theorem on iterated
limits from \cite[Chap.~2, \S~4]{K} yields that the net
$(\mu_{(t,\beta)})_{(t,\beta)\in D}$ belongs to~$\mathbb
M(\mathbf{A},\mathbf{a},\mathbf{g})$ and converges $\mathbf
A$-vaguely to~$\gamma_0$. Thus, $\gamma_0\in\mathcal
M(\mathbf{A},\mathbf{a},\mathbf{g})$ as was to be proved.\end{proof}

\begin{corollary}\label{cor:WS}
Every minimizing measure $\lambda=\lambda_{\mathbf{A}}$ (if exists)
is extremal, i.\,e.,
\begin{equation}\label{WS}\mathfrak S(\mathbf{A},\mathbf{a},\mathbf{g})\subset
\mathfrak E(\mathbf{A},\mathbf{a},\mathbf{g}).\end{equation}\end{corollary}

\begin{proof} For every sufficiently large
$\mathbf{K}\in\{\mathbf{K}\}_{\mathbf{A}}$, we define
$\hat{\lambda}_{\mathbf{K}}$ by~(\ref{hatmu}) with~$\lambda$ instead
of~$\mu$. Then
$(\hat{\lambda}_{\mathbf{K}})_{\mathbf{K}\in\{\mathbf{K}\}_{\mathbf{A}}}$
belongs to $\mathbb M(\mathbf{A},\mathbf{a},\mathbf{g})$, which is
clear from~(\ref{4w}) with~$\mu$ replaced by~$\lambda$. On the other
hand, this net converges $\mathbf A$-vaguely to~$\lambda$; hence,
$\lambda\in\mathcal M(\mathbf{A},\mathbf{a},\mathbf{g})$. Combined
with~(\ref{WM}), this completes the proof.
\end{proof}

\subsection{Central lemma} The
following assertion will be central in the proof of
Theorem~\ref{exist}.

\begin{lemma}\label{lemma:exist}
Fix $i\in I$ and assume that either $g_{i,\sup}<\infty$, or
{\rm(\ref{growth})} holds for some $r_i\in(1,\infty)$ and
$\omega_i\in\mathcal E$. If, moreover, $A_i$ either is compact or
has finite interior capacity, then
\begin{equation}
\int g_i\,d\gamma^i=a_i\quad\mbox{for all \ }\gamma\in\mathfrak
E(\mathbf{A},\mathbf{a},\mathbf{g}). \label{24}
\end{equation}\end{lemma}

\begin{proof}Given $\gamma\in\mathfrak E(\mathbf{A},\mathbf{a},\mathbf{g})$,
choose a net $(\mu_s)_{s\in S}\in\mathbb
M(\mathbf{A},\mathbf{a},\mathbf{g})$ that converges to~$\gamma$ both
strongly and $\mathbf A$-vaguely. Taking a subnet if necessary, one
can certainly assume $(\mu_s)_{s\in S}$ to be strongly bounded.

Of course, (\ref{24}) needs to be proved only if the set~$A_i$ is
noncompact; then its capacity has to be finite. Hence,
by~\cite[Th.~4.1]{F1}, for every $E\subset A_i$ there exists a
measure $\theta_E\in\mathcal E^+(\,\overline{E}\,)$, called an
interior equilibrium measure associated with~$E$, which admits the
properties
\begin{equation}
\theta_E(\mathbf X)=\|\theta_E\|^2=C(E), \label{5}
\end{equation}
\begin{equation}
\kappa(x,\theta_E)\geqslant1\quad\mbox{n.\,e. in \ } E. \label{6}
\end{equation}

Also observe that there is no loss of generality in assuming $g_i$
to satisfy~(\ref{growth}) for some $r_i\in(1,\infty)$ and
$\omega_i\in\mathcal E$. Indeed, otherwise $g_i$ has to be bounded
from above (say by $M_1$), which combined with~(\ref{6}) again
gives~(\ref{growth}) for $\omega_i:=M_1^{r_i}\,\theta_{A_i}$,
$r_i\in(1,\infty)$ being arbitrary.

To establish (\ref{24}), we shall treat $A_i$ as a locally compact
space with the topology induced from~$\mathrm X$. Let $\chi_E$
denote the characteristic function of a given set $E\subset A_i$ and
let $E^c:=A_i\setminus E$. Further, let $\{K_i\}$ be the increasing
family of all compact subsets~$K_i$ of~$A_i$. Since $g_i\chi_{K_i}$
is upper semicontinuous on~$A_i$ while $(\mu_s^i)_{s\in S}$
converges to~$\gamma^i$ vaguely, from Lemma~\ref{lemma:lower} we get
\[
\int g_i\,\chi_{K_i}\,d\gamma^i\geqslant\limsup_{s\in S}\int
g_i\,\chi_{K_i}\,d\mu_s^i\quad\mbox{for every \ } K_i\in\{K_i\}.\]
On the other hand, application of Lemma~1.2.2 from~\cite{F1} yields
\[
\int g_i\,d\gamma^i=\lim_{K_i\in\{K_i\}}\,\int
g_i\,\chi_{K_i}\,d\gamma^i.\] Combining the last two relations, we
obtain
\[
a_i\geqslant\int g_i\,d\gamma^i\geqslant \limsup_{(s,\,K_i)\in
S\times\{K_i\}}\,\int g_i\,\chi_{K_i}\,d\mu_s^i=
a_i-\liminf_{(s,\,K_i)\in S\times\{K_i\}}\,\int
g_i\,\chi_{K^c_i}\,d\mu_s^i,\] $S\times\{K_i\}$ being the directed
product of the directed sets~$S$ and~$\{K_i\}$. Hence, if we prove
that
\begin{equation}
\liminf_{(s,\,K_i)\in S\times\{K_i\}}\,\int
g_i\,\chi_{K^c_i}\,d\mu_s^i=0, \label{25}
\end{equation}
the desired relation~(\ref{24}) follows.

Consider an interior equilibrium measure $\theta_{K^c_i}$, where
$K_i\in\{K_i\}$ is given. Then application of~Lemma~4.1.1 and
Theorem~4.1 from~\cite{F1} shows that
\[
\|\theta_{K^c_i}-\theta_{\tilde{K}^c_i}\|^2\leqslant
\|\theta_{K^c_i}\|^2-\|\theta_{\tilde{K}^c_i}\|^2\quad\mbox{provided
\ }K_i\subset\tilde{K}_i.\] Furthermore, it is clear from~(\ref{5})
that the net $\|\theta_{K^c_i}\|$, $K_i\in\{K_i\}$, is bounded and
nonincreasing, and hence fundamental in $\mathbb R$. The preceding
inequality thus shows that the net
$(\theta_{K^c_i})_{K_i\in\{K_i\}}$ is strongly fundamental in
$\mathcal E$. Since, clearly, it converges vaguely to zero, the
property~(C$_1$) (see.~Sec.~\ref{sec:2}) implies immediately that
zero is also one of its strong limits and, hence,
\begin{equation}
\lim_{K_i\in\{K_i\}}\,\|\theta_{K^c_i}\|=0. \label{27}
\end{equation}

Write $q_i:=r_i(r_i-1)^{-1}$, where $r_i\in(1,\infty)$ is a number
involved in condition~(\ref{growth}). Combining (\ref{growth}) with
(\ref{6}) shows that the inequality
\[
g_i(x)\,\chi_{K^c_i}(x)\leqslant\kappa(x,\omega_i)^{1/r_i}\,
\kappa(x,\theta_{K^c_i})^{1/q_i}\]  subsists n.\,e.~in~$A_i$, and
hence $\mu_s^i$-almost everywhere in~$A_i$ by virtue of~Lemma~2.3.1
from~\cite{F1} and the fact that $\mu_s^i$ is a measure of finite
energy, compactly supported in~$A_i$. Having integrated this
relation with respect to $\mu_s^i$, we then apply the H\"older and,
subsequently, the Cauchy-Schwarz inequalities to the integrals on
the right. This gives
\begin{eqnarray*}
\int g_i\,\chi_{K^c_i}\,d\mu_s^i&\!\!\!\leqslant\!\!\!&
\Bigl[\int\kappa(x,\omega_i)\,d\mu_s^i(x)\Bigr]^{1/r_i}\,
\Bigl[\int\kappa(x,\theta_{K^c_i})\,d\mu_s^i(x)\Bigr]^{1/q_i}\\
{}
&\!\!\!\leqslant\!\!\!&\|\omega_i\|^{1/r_i}\,\|\theta_{K^c_i}\|^{1/q_i}\,\|\mu_s^i\|.
\end{eqnarray*}
Taking limits here along $S\times\{K\}$ and using (\ref{7.1i}) and
(\ref{27}), we obtain~(\ref{25}) as desired.
\end{proof}

\begin{corollary}\label{cor:j} Assume that for every $i\in I$ either
$g_{i,\sup}<\infty$, or {\rm(\ref{growth})} holds for some
$r_i\in(1,\infty)$ and $\omega_i\in\mathcal E$. If moreover $\kappa$
is strictly positive positive (hence, perfect), while
$C(A_j)<\infty$ for some $j\in I$, then ${\rm
cap}\,(\mathbf{A},\mathbf{a},\mathbf{g})$ is finite as
well\,\footnote{Cf.~Lemma~\ref{lemma:finite}.}.
\end{corollary}

\begin{proof}In consequence of Lemma~\ref{lemma:exist}, every extremal measure~$\gamma$
(whose existence has been ensured by Lemma~\ref{lemma:WM}) is
nonzero; hence, due to the strict positive definiteness of the
kernel, $\|\gamma\|^2\ne0$. When combined with~(\ref{extremal}),
this yields $\|\mathcal
E(\mathbf{A},\mathbf{a},\mathbf{g})\|^2\ne0$, as was to be proved.
\end{proof}

\section{Proof of Theorems~\ref{exist} and~\ref{cor:cont}}\label{sec:proof.th.str}

Basing on the results of Sec.~\ref{sec:extremal}, we are now in a
position to complete the proofs of Theorems~\ref{exist}
and~\ref{cor:cont}.

Assume that $\kappa$, $\mathbf A$, $\mathbf a$, and $\mathbf g$
satisfy all the restrictions of Theorem~\ref{exist}. Fix an
arbitrary $\gamma\in\mathfrak
E(\mathbf{A},\mathbf{a},\mathbf{g})$~--- it exists due
to~Lemma~\ref{lemma:WM},~(i). Then, by Lemma~\ref{lemma:exist},
(\ref{24})~holds for every $i\in I$ and, consequently,
$\gamma\in\mathcal E(\mathbf{A},\mathbf{a},\mathbf{g})$. In view
of~(\ref{extremal}), this shows that $\gamma$ is actually a
minimizer in the $\mathcal
E(\mathbf{A},\mathbf{a},\mathbf{g})$-problem. Hence, the $\mathcal
E(\mathbf{A},\mathbf{a},\mathbf{g})$-problem is solvable and,
moreover, $\mathfrak
E(\mathbf{A},\mathbf{a},\mathbf{g})\subset\mathfrak
S(\mathbf{A},\mathbf{a},\mathbf{g})$. When combined with~(\ref{WS}),
this gives $\mathfrak E(\mathbf{A},\mathbf{a},\mathbf{g})=\mathfrak
S(\mathbf{A},\mathbf{a},\mathbf{g})$. Therefore
Lemma~\ref{lemma:WM},~(iii) yields the $\mathbf A$-vague compactness
of $\mathfrak S(\mathbf{A},\mathbf{a},\mathbf{g})$, and the proof of
Theorem~\ref{exist} is complete. Because of~(\ref{WM}), the last
identity also implies
\begin{equation}\label{SWM}\mathfrak S(\mathbf{A},\mathbf{a},\mathbf{g})=
\mathfrak E(\mathbf{A},\mathbf{a},\mathbf{g})= \mathcal
M(\mathbf{A},\mathbf{a},\mathbf{g}).
\end{equation}

To prove Theorem~\ref{cor:cont}, for every
$\mathbf{K}\in\{\mathbf{K}\}_{\mathbf{A}}$ fix an arbitrary
$\lambda_{\mathbf{K}}\in\mathfrak
S(\mathbf{K},\mathbf{a},\mathbf{g})$~--- its existence has been
ensured by Corollary~\ref{cor:exist}. According to
Lemma~\ref{lemma.cont},
\begin{equation}\label{kmin}
(\lambda_{\mathbf{K}})_{\mathbf{K}\in\{\mathbf{K}\}_{\mathbf{A}}}\in\mathbb
M(\mathbf{A},\mathbf{a},\mathbf{g}).\end{equation} Therefore,
by~(\ref{SWM}), every $\mathbf A$-vague cluster point
of~$(\lambda_{\mathbf{K}})_{\mathbf{K}\in\{\mathbf{K}\}_{\mathbf{A}}}$
is an element of~$\mathfrak S(\mathbf{A},\mathbf{a},\mathbf{g})$,
which is a part of the desired conclusion.

What is left is to show that
$\lambda_{\mathbf{K}}\to\lambda_{\mathbf{A}}$ strongly, where
$\lambda_{\mathbf{A}}\in\mathfrak
S(\mathbf{A},\mathbf{a},\mathbf{g})$ is arbitrarily given, but this
is obtained directly from~(\ref{SWM}), (\ref{kmin}), and
Lemma~\ref{lemma:WM},~(ii).

\bigskip
\flushleft

{\small Institute of Mathematics\\
National Academy of Sciences of Ukraine\\
3 Tereshchenkivska Str.\\
01601, Kyiv-4, Ukraine\\
e-mail: natalia.zorii@gmail.com}


\begin{thebibliography}{00}

\bibitem{B1}
N.~Bourbaki, {\it Topologie g\'en\'erale, Chap.~I--II\/},
Actualit\'es Sci. Ind., 1142, Paris (1951).

\bibitem{B2}
N.~Bourbaki, {\it Int\'egration, Chap.~I--IV\/}, Actualit\'es Sci.
Ind., 1175, Paris (1952).

\bibitem{B3}
N.~Bourbaki, {\it Int\'egration des measures\/}, Actualit\'es Sci.
Ind., 1244, Paris (1956).

\bibitem{Car} H.~Cartan, {\it Th\'eorie du potentiel newtonien:
\'energie, capacit\'e, suites de potentiels\/}, Bull. Soc. Math.
France {\bf 73} (1945), 74--106.

\bibitem{D1} J.~Deny, {\it Les potentiels d'\'energie finite\/}, Acta Math. {\bf 82} (1950), 107--183.

\bibitem{D2} J.~Deny, {\it Sur la d\'{e}finition de
l'\'{e}nergie en th\'{e}orie du potential\/}, Ann. Inst. Fourier
Grenoble {\bf 2} (1950), 83--99.

\bibitem{E1}
R.~Edwards, {\it Cartan's balayage theory for hyperbolic Riemann
surfaces\/}, Ann. Inst. Fourier {\bf 8} (1958), 263--272.

\bibitem{E2}
R.~Edwards, {\it Functional analysis. Theory and applications\/},
Holt. Rinehart and Winston, New York (1965).

\bibitem{F1} B.~Fuglede, {\it On the theory of potentials in
locally compact spaces\/}, Acta Math.~{\bf 103} (1960), 139--215.

\bibitem{F2} B.~Fuglede, {\it Caract\'erisation des noyaux consistants
en th\'eorie du potentiel\/}, Comptes Rendus~{\bf 255} (1962),
241--243.

\bibitem{GR0}A.\,A.~Gonchar, E.\,A.~Rakhmanov, {\it On convergence
of simultaneous Pad\'{e} approximants for systems of functions of
Markov type\/}, Number theory, Mathematical analysis and
Applications, Trudy Mat. Inst. Steklov, Nauka, Moscow {\bf 157}
(1981), 31--48 (in Russian); English transl.~in: Proc. Steklov Inst.
Math. {\bf 157}, Amer. Math. Soc., Providence, RI (1983).

\bibitem{GR}
A.\,A.~Gonchar, E.\,A.~Rakhmanov, {\it On the equilibrium problem
for vector potentials\/}, Uspekhi Mat. Nauk {\bf 40}:4 (1985),
155--156; English transl.~in: Russian Math. Surveys {\bf 40}:4
(1985), 183--184.

\bibitem{H}
W.\,K.~Hayman, {\it Multivalent functions\/}, Cambridge Tracts in
Mathematics, {\bf 110}, Cambridge University Press, Cambridge
(1994).

\bibitem{HK}
W.\,K.~Hayman, P.\,B.~Kennedy, {\it Subharmonic functions\/},
Academic Press, London (1976).

\bibitem{K}
J.\,L.~Kelley, {\it General topology\/}, Princeton, New York (1957).

\bibitem{L}
N.\,S.~Landkof, {\it Foundations of modern potential theory\/},
Nauka, Fizmatlit, Moscow (1966); English trans., Springer--Verlag,
Berlin (1972).

\bibitem{MS} E.\,H.~Moore, H.\,L.~Smith, {\it A general theory of
limits\/}, Amer. J. Math. {\bf 44} (1922), 102--121.

\bibitem{NS} E.\,M.~Nikishin, V.\,N.~Sorokin, {\it Rational approximations and orthogonality\/},
Nauka, Fizmatlit, Moscow (1988); English trans., Translations of
Mathematical Monographs {\bf 44}, Amer. Math. Soc., Providence, RI
1991.

\bibitem{O}
M.~Ohtsuka, {\it On potentials in locally compact spaces\/},
J.~Sci.~Hiroshima Univ. Ser.~A-1 {\bf 25} (1961), 135--352.

\bibitem{ST} E.\,B.~Saff, V.~Totik, {\it Logarithmic potentials
with external fields\/}, Sprin\-ger--Verlag, Berlin (1997).

\bibitem{Z2} N.~Zorii, {\it A noncompact variational problem in the
Riesz potential theory.~I; II\/}, Ukrain. Math.~Zh. {\bf 47} (1995),
1350--1360; {\bf 48} (1996), 603--613 (in Russian); English
transl.~in: Ukrain. Math.~J. {\bf 47} (1995); {\bf 48} (1996).

\bibitem{Z3} N.~Zorii, {\it Extremal problems in the theory of
potentials in locally compact spaces.~I; II; III\/}, Bull. Soc. Sci.
Lettr. \L\'od\'z {\bf 50} S\'er. Rech. D\'eform. {\bf 31} (2000),
23--54; 55--80; 81--106.

\bibitem{Z4} N.~Zorii, {\it On the solvability of the Gauss
variational problem\/}, Comput. Meth. Funct. Theory {\bf 2} (2002),
427--448.

\bibitem{Z6} N.~Zorii, {\it Necessary and sufficient conditions for
the solvability of the Gauss variational problem\/}, Ukrain.
Math.~Zh. {\bf 57} (2005), 60--83 (in Russian); English transl.~in:
Ukrain. Math.~J. {\bf 57} (2005).

\bibitem{Z.arh} N.~Zorii, {\it Interior capacities of condensers with infinitely many plates
in a locally compact space\/}, arXiv:0906.4522.

\end{thebibliography}
\end{document}